\pdfoutput=1
\documentclass[a4paper,american,floatfix,pdftex,superscriptaddress,
citeautoscript,
final]{revtex4-2}%
\usepackage[T1]{fontenc}
\usepackage{graphicx}%
\usepackage[utf8]{inputenc}
\usepackage{hyperref, hypernat}
\usepackage[todos]{CMImacros}
\usepackage[boxed]{algorithm2e} 
\usepackage{amsfonts,amsmath,amssymb,amsthm}%
\usepackage{mathtools}
\usepackage{tocloft}

\newtheorem{definition}{Definition}

\newtheorem{thm}{Theorem}

\theoremstyle{thm}

\graphicspath{{./figures/}} 

\newcommand{\cfeldesy}{\affiliation{Center for
		Free-Electron Laser Science CFEL, Deutsches	Elektronen-Synchrotron 
		DESY, Notkestraße 85, 22607 Hamburg, Germany}}%
\newcommand{\uhhcui}{\affiliation{Center for Ultrafast Imaging, Universität 
		Hamburg, Luruper
		Chaussee 149, 22761 Hamburg, Germany}}%
\newcommand{\uhhphys}{\affiliation{Department of Physics, Universität 
		Hamburg, Luruper Chaussee 149,
		22761 Hamburg, Germany}}%
\newcommand{\uhhmaths}{\affiliation{Department of Mathematics, 
		Universität Hamburg, Bundesstraße 55,
		20146, Hamburg, Germany}}%
\newcommand{\ysemail}{\email[Email: ]{yahya.saleh@cfel.de}}%
\newcommand{\cmiweb}{\homepage[URL:~]{https://www.controlled-molecule-imaging.org}}%

\SetAlCapSkip{\abovecaptionskip}
\begin{document}
\title{Augmenting Basis Sets by Normalizing Flows}

	\author{Yahya Saleh}\ysemail \cmiweb \uhhmaths\cfeldesy %
\author{Armin Iske}\uhhmaths%
\author{Andrey Yachmenev}\cfeldesy\uhhcui%
\author{Jochen Küpper}\cfeldesy\uhhcui\uhhphys%
\begin{abstract}
   Approximating functions by a linear span of truncated basis sets is a standard procedure for the
   numerical solution of differential and integral equations. Commonly used concepts of
   approximation methods are well-posed and convergent, by provable approximation orders. On the
   down side, however, these methods often suffer from the curse of dimensionality, which limits
   their approximation behavior, especially in situations of highly oscillatory target functions.
   Nonlinear approximation methods, such as neural networks, were shown to be very efficient in
   approximating high-dimensional functions.
   We investigate nonlinear approximation methods that are constructed by composing standard basis
   sets with normalizing flows. Such models yield richer approximation spaces while maintaining the
   density properties of the initial basis set, as we show. Simulations to approximate
   eigenfunctions of a perturbed quantum harmonic oscillator indicate convergence with respect to
   the size of the basis set.
\end{abstract}
	\maketitle

\section{Introduction}
Numerical approximation schemes arise in many relevant problems that deal 
with solutions of
(partial) differential and integral equations~\cite{Iske:ApproxTh:2018}, where commonly used
approximation methods rely on fix point iterations. Another popular class of approximation methods
expands the sought target function by a linear span from a (countable) set of basis functions. In
fact, when applied to the solution differential equations, this standard approach leads to spectral
methods.

Spectral methods were studied quite
extensively~\cite{Gottlieb:SpectralMethods:1977}, where they enjoyed
increasing popularity in various applications from computational science and
engineering, especially for their good approximation properties.
Despite their favorable approximation properties, their convergence rate
degrades exponentially at increasing problem dimension. This phenomenon,
referred to as {\em the curse of dimensionality}, is due to the linearity of the
approximation method.
This leads to severe limitations, \eg,  in applications of quantum mechanics and
dynamics, where the systems of interest are inherently high-dimensional and
where approximations of large numbers of eigenfunctions to unbounded linear
operators are required.

Nonlinear models, such as neural networks, have recently been explored as alternative approximation methods for solving (partial) differential equations~\cite{E:CA55:369}.
Their efficiency in approximating high-dimensional functions for challenging applications, ranging from image recognition to natural language processing, shows their potential for solving high-dimensional differential equations, in particular for variational problems.
One important problem class are infinite dimensional eigenvalue problems,
\eg,  Schrödinger equations, since they are strongly related to variational
simulations of numerous physics phenomena and, moreover, they often
demand solutions for a large number of eigenvalues.

Indeed, neural networks were successfully applied to various finite-~\cite{Carleo:Science355:602} and infinite-dimensional~\cite{Hermann:Natchem12:891} quantum systems,
yielding higher accuracies at a lower computational scaling, in contrast to traditional methods.
Using standard neural network architectures, such as multilayer perceptrons, however, reduces their reliability, due to the lack of analytic convergence results.
In addition, their concepts are widely restricted to modeling eigenfunctions
corresponding to states with low energies.

In the physics applications that we have in mind, we are primarily interested in nonlinear
approximation schemes, where standard basis sets are composed with {\em normalizing
   flows}~\cite{Cranmer:arXiv1904:05903}, \ie, invertible parameterizable functions \footnote{In the
   original paper~\cite{Cranmer:arXiv1904:05903} such approximating schemes were proposed to solve
   Schrödinger equations, where they were referred to as \emph{quantum flows}. We recognize the
   applicability of such models in approximation problems unrelated to differential equations or
   quantum mechanics, and therefore refrain from using this terminology.
   Instead, we propose to refer to
   such models as \emph{augmented basis sets} since normalizing flows increase/augment the
   expressivity of standard basis sets.}.
The aim of such a model is to increase the expressivity of standard basis sets, thereby enabling a more diverse and expansive approximation space.
We restrict ourselves to the relevant case of basis sets of $L^2$.
We show that augmented basis sets, \ie, standard $L^2$ basis sets composed
with invertible parametrizable mappings,
lead to a family of linear spaces that are all dense in $L^2$, whereby they generate much richer approximation spaces.

We use invertible residual neural networks~\cite{Behrmann:ICML2019:573} to augment the expressivity of Hermite functions,
before we apply them to find the eigenfunctions of a perturbed quantum harmonic oscillator.
Our results show fast convergence with respect to the size of the augmented basis. Moreover, we
argue that the inductive bias attained by the use of an initial basis set allows for
the approximation
of a larger number of eigenfunctions than what is possible with more flexible
models, such as standard neural networks.

\section{Augmenting Basis Sets}
For an open domain $\Omega \subseteq \mathbb{R}^d$, let $L^2(\Omega)$ be
the linear space of square integrable functions from $\Omega$ to
$\mathbb{R}^d$.
From now, we drop the dependence on $\Omega$ for notational simplicity.
Moreover, let $\{\phi_n\}_n$ denote a basis set of $L^2$ with inner product $\langle., . \rangle$.
Let $g_\theta: \Omega \longrightarrow \Omega$ be a smooth
parametrizable bijection.
Finally, let '$\circ$' denote the standard composition between functions.

\begin{definition}[Augmented set of functions]\label{def:asf}
On the above assumptions, we define an augmented set of functions $\{ \phi^A_n \}_n$ by
$$
    \phi_n^A(x) := (\underbrace{\phi_n \circ
    g^{-1}_\theta}_{:=\tilde{\phi}_n})(x) |\det \nabla_x g^{-1}_\theta|^{1/2}.
$$
For simplicity, we drop the dependence of $g$ on $\theta$,
and we define the weighted $L^2_{g ^{-1}}$ space induced by the bijection $g$ as
\begin{equation*}
	L^2 _{g^{-1}} := \left\{
	f  \, \text{ : } \, \int |f|^2 |\det \nabla_x
	g^{-1}| dx
	=
		\int |f|^2
	dg^{-1} < \infty \right\},
\end{equation*}
with the usual equivalence relation in $L^p$-spaces. We denote by $\langle ., .
\rangle_{g^{-1}}$
the inner product on the weighted space $L^2_{g ^{-1}}$, \ie,
\begin{equation*}
   \langle f, h
   \rangle_{g^{-1}} = \int f(x) h(x) |\det \nabla_x g^{-1}| dx.
\end{equation*}
\end{definition}

The following theorem shows that any augmented set of functions, as in Definition~\ref{def:asf}, is a basis set.
This observation enables us to construct approximation spaces with good convergence properties.
\begin{thm}[Augmented basis sets]
   \label{lem:basis}
   The augmented set of functions $\{\tilde{\phi}_n\}_n$ is an orthonormal basis of $L^2_{g^{-1}}$.
\end{thm}

\begin{proof}
   The orthonormality of the functions in $\{\tilde{\phi}_n\}_n$ can readily be seen by a simple change of variable.
   To prove that $\{\tilde{\phi}_n\}_n$ is a basis set for $L^2_{g^{-1}}$, we take one $f \in L^2_{g^{-1}}$ satisfying
   $f \perp \tilde{\phi}_n$ for all $n$. In this case, we have the orthogonality
   \begin{equation*}
      0 = \langle \tilde{\phi}_n, f \rangle_{g^{-1}} = \int \phi_n(y) f(g(y)) dy
      \qquad \mbox{ for all } n.
   \end{equation*}
   Since $\{\phi_n\}_n$ is a basis of $L^2$, we have  $f\circ g\equiv 0$ , which is
   equivalent to $\text{ran} (f \circ g) = \{0\}$. Given that $g^{-1}$ is invertible we have
   $\text{range}(g) = \text{domain}(g^{-1}) =\Omega$. Hence, $\text{range}(f
   \circ g) =\text{range}(f ) $ and
   thus $f \equiv 0$.
\end{proof}

Note that $\{\tilde{\phi}_n\}_n$ being a basis of $L^2_{g^{-1}}$ is equivalent to $\{\phi^A_n\}_n$ being a basis of $L^2$.
Also note that $\overline{\text{span} \{\phi_n^A\}}$ depends on $\theta$.

In our numerical investigations, we work with the set of Hermite functions $\{\phi_n\}_n$.
In one dimension these are defined by
\begin{equation*}
   \phi_n (x) := h_n(x) \exp(-x^2/2),
\end{equation*}
where $h_n$ denotes the $n$th Hermite polynomial. Note that this set of functions is a basis of $L^2(\mathbb{R}) $.

Now there are many ways to construct a parametrizable bijection $g$.
As these bijections emerged in the machine learning community to perform
generative modeling, the most common framework for constructing such
bijections
is \emph{via} special types of neural networks. Here we use
invertible residual networks~\cite{Behrmann:ICML2019:573} having the form
$g(x) = x + k(x)$, where $k$ is a standard neural network with linear
layers and nonlinear
activations. It was shown that this model is invertible, if
$\text{Lip}(k)<1$~\cite{Behrmann:ICML2019:573}, where $\text{Lip}(k)$ is
the Lipschitz constant of $k$. To satisfy the invertibility condition for a
residual network $k$,
we employed Lipschitz-continuous  nonlinear activation functions and divided the linear layer weight matrices
by their spectral norm.

In the next section, we show how to numerically solve the Schrödinger equation for a perturbed quantum harmonic oscillator,
in order to demonstrate the advantages of approximating $L^2$ functions with the augmented Hermite basis.

\section{Numerical Investigation and Discussion}

Let $\Omega := \mathbb{R}^d$.
We consider approximating the eigenvalues of quantum mechanical Hamiltonian operators
 $H: D(H) \longrightarrow L^2$, \ie, we aim at finding the eigenpairs $(E_k, \psi_k)$ that
solve the Schrödinger equation
\begin{equation}
   \label{eq:Schroedinger}
   H \psi_k
   =
   E_k \psi_k,
   \quad \mbox{ where } \int |\psi_k|^2 dx = 1
   \qquad \mbox{ for } k=0, 1, \ldots.
\end{equation}
To this end, we use a Bubnov-Galerkin numerical scheme
\cite{Gottlieb:SpectralMethods:1977}.
We assume that $H$ is a self-adjoint operator, where we consider the generic case of $H= T+V$
with $T=-\frac{1}{2} \Delta$ denoting the kinetic-energy operator and $V$
denotes the scalar potential-energy function, \ie, $(Vf)(x) = V(x)f(x)$.
Following standard theories of quantum mechanics \cite{Lubich:QCMD}, we set $D(H) = \mathcal{H} ^2$ for the Sobolev space of functions,
which are  square integrable along with their generalized partial derivatives up to second order.

We discretize problem~(\ref{eq:Schroedinger}) by using a truncated series of
Hermite functions
$\{\phi_n\}_{n=0}^{N-1}$ and a truncated series of augmented Hermite functions
$\{\phi_n^A\}_{n=0}^{N-1}$, with the aim to compare the convergence
properties of the two schemes as $N$ increases.
Projecting
\eqref{eq:Schroedinger} on the linear space of the two basis sets yields
\begin{equation}
   \label{eq:Schroedinger_weak}
   \tilde{H} C_{n} = \tilde{E_{n}} C_n
   \qquad \mbox{ for } n=0, \dots, N-1,
\end{equation}
where $\tilde{H}$ is an $N \times N$ matrix whose $(i,j)$-th entry is
$\tilde{H}[i,j] = \langle \phi_i , H \phi_j \rangle$ or $\langle \phi_i ^A, H \phi_j^A \rangle$, for Hermite and augmented Hermite
basis sets, respectively.
Moreover, $C_n$ in~(\ref{eq:Schroedinger_weak}) is a vector of size $N$.
Hence, solving~\eqref{eq:Schroedinger} boils down to solving the finite dimensional eigenvalue problem
\eqref{eq:Schroedinger_weak}. Note that the approximation space in the case of using Hermite
functions reads $V = \text{span}(\phi_0, \dots, \phi_{N-1})$. In the case of augmented Hermite functions,
the approximation space $V^A$ depends on parameter $\theta$, where
$$
     V^A = \{\bigcup\limits_{\theta} \ \text{span}(\phi^A_0, \dots, \phi^A_{N-1}) \}
$$
yields a richer approximation space.

To solve problem \eqref{eq:Schroedinger_weak}, we used a direct eigensolver to find the  eigenvector coefficients and eigenvalues
$C_n, \tilde{E}_n$, for $n=0, \dots , N-1$. For augmented Hermite bases, the
resulting eigenvalues $\tilde{E}_n$ and eigenvectors $C_n$ depend on the parameters $\theta$ of the normalizing flow.
Noting that
$$
   \sum_{n=0}^{N-1} \tilde{E}_n \geq \sum_{n=0}^{N-1} E_n,
$$
we used a first-order iterative
optimization algorithm to optimize the parameters $\theta$ by minimizing the loss function
$$
     \mathcal{L} = \sum_{n=0}^{N-1} \tilde{E}_n.
$$
For more details on the training procedure, we refer to \autoref{sec:supplementary}.

We fixed $d=1$ and solved \eqref{eq:Schroedinger} with the anharmonic
potential $V(x) =
\frac{1}{2}x^2 + \frac{1}{4}x^4$. We
 studied the convergence of both approximation schemes as a function of the
truncation parameter, $N$. We considered values of $N$ ranging
from $1$ to $49$. To compute the elements of $\tilde{H}$ we
used Gauss-Hermite quadratures see
\autoref{sec:supplementary}.
\autoref{fig:convergence} shows five bands, each of which
represents the sum of five eigenvalues in ascending order.
The graph clearly shows  that both Hermite and augmented Hermite basis functions converge
as a function of $N$, with the latter achieving faster convergence.
Furthermore, because a Bubnov-Galerkin numerical scheme for a Schrödinger
equation based on Hermite functions
converges to the true eigenvalues of \eqref{eq:Schroedinger}, we conclude that a numerical scheme based on augmented Hermite functions converges to the true solutions.
Finally, our proposed numerical scheme enables us to compute many
eigenvalues at once.
We remark that this cannot be achieved by applying straight forward standard neural networks.
We attribute this to the good inductive bias provided by the use of an $L^2$
basis, which constitutes of the quantum harmonic oscillator eigenfunctions.

\autoref{fig:loss_function} shows the convergence of the total loss function, \ie,
the sum of all eigenvalues  for $N = 5,6,7,8,9$ at each training iteration $t$, plotted for
two discretization schemes.
We can see that the loss function
converges faster for smaller $N$ implying that a larger number of
training iterations are required to model a larger number of excited states.
The loss function converges rapidly for low-lying states with only few iterations necessary
to optimize the nonlinear parameters $\theta$. This
demonstrates the high quality of the inductive bias
provided by Hermite functions.

Finally, we analyze the convergence rate of both discretization schemes numerically.
To this end, we use the concept of {\em Q-convergence}, see~\cite{Sun:OptimizationTh:2006} for details:
a sequence $(x_N)_N$ converging to a limit $x^*$ is said to converge {\em
Q-super-linearly}
to $x^*$, iff
\begin{equation*}
   e_N :=
   \frac{|x_N -
         x^*|}{|x_{N-1}-x^*|} \longrightarrow 0
         \quad \mbox{ for } N \to \infty.
\end{equation*}
To define the sequence $(e_N)_N$ for our schemes, we use  the second band of
eigenvalues (states from 5 to 10) for both discretization schemes, \ie,
$$
    x_N = \sum_{n=5}^{10} \tilde{E}_{n,N}.
$$
The sequence convergence limit is defined as
$$
    x^* = \sum_{n=5}^{10} E^*_{n},
$$
where the reference converged eigenvalues $E^*_n$ were computed using a large basis set with $N=29$.
The reference eigenvalues computed with Hermite and augmented Hermite basis sets differ slightly
(with $3 \times 10^{-3}$ at maximum)
which can be attributed to the properties of the stochastic nonlinear optimizer. Therefore,
we choose to define $E^*_n$ based on the converged eigenvalues computed by the respective
discretization scheme.

We computed
the quantity $e_N$ for different values of $N$. Since the resulting sequence
values
$\{e_N\}_{N=10}^{N=29}$ are noisy, we fitted a linear regressor to them and
plotted the resulting function
for the two discretization schemes in \autoref{fig:convergence_rate}.
Both discretization schemes result in Q-super-linear convergence,
with augmented Hermite converging faster.
We observed that the convergence of augmented Hermite functions worsens as the accuracy increases,
whereas this behavior is absent when discretizing with Hermite functions.
We attributed this to the iterative nonlinear optimization performed for augmented Hermite functions,\
which resulted in the oscillating behavior of the loss function around its
extremal points.

For the same size of the linear expansion $N$, the computational costs for
performing the numerical simulation above are,
indeed, higher when using augmented Hermite functions. This is due to the
need for training the normalizing flow and having to compute the trace of the
projected Hamiltonian at each step. However, clearly a smaller $N$ is required 
to converge a certain number of exited states when using augmented Hermite 
functions than when using Hermite functions.
\begin{figure}
      \includegraphics[width=0.5\linewidth]{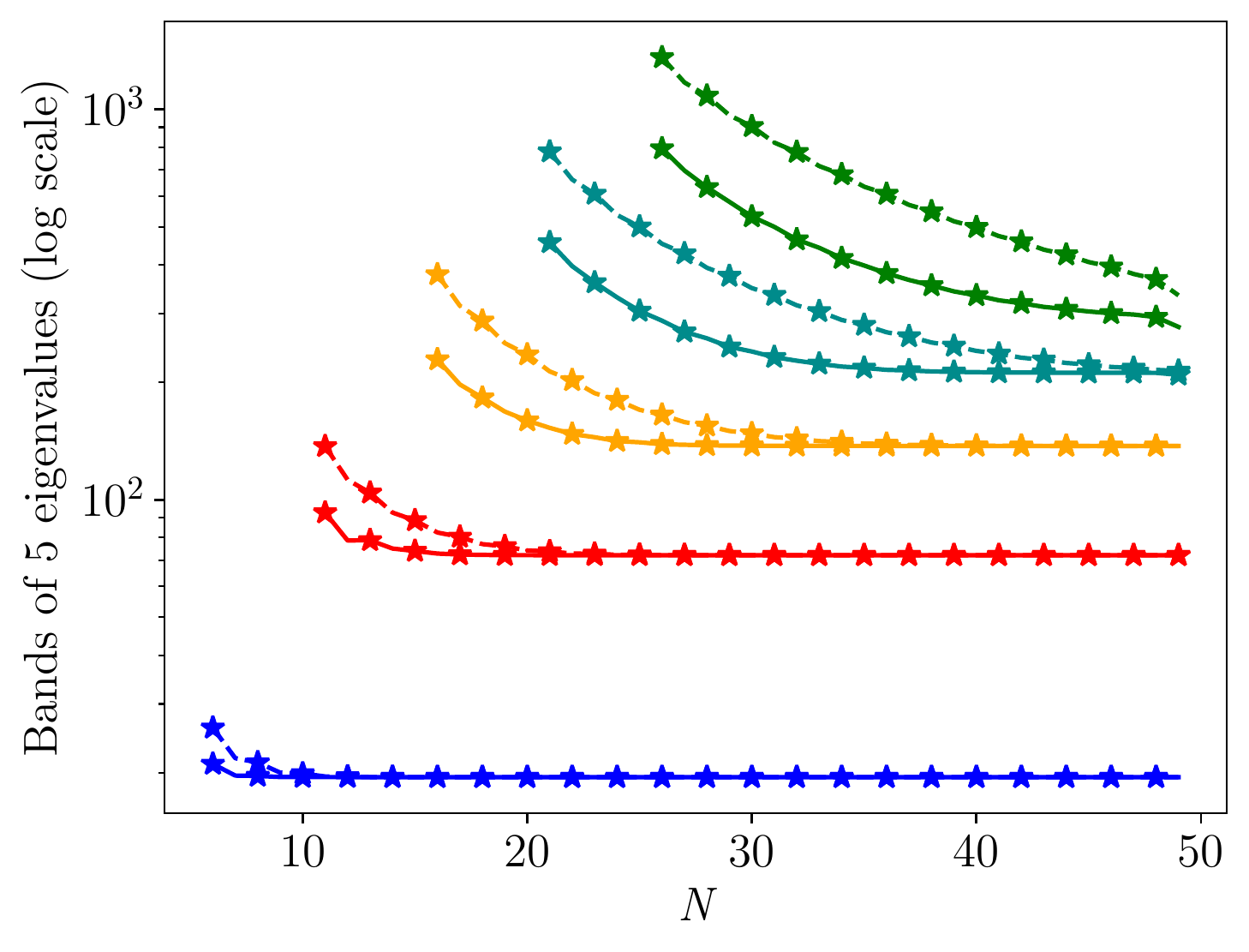}%
      \caption{The average error over bands of five approximate eigenvalues is
      depicted as a function of the basis truncation parameter $N$.
      The results are plotted for Hermite (dotted lines) and augmented Hermite (solid lines) discretizations schemes.
      The color of the lines indicates different energy bands.
      To improve the contrast between the lines, every other marker is omitted. }
\label{fig:convergence}
\end{figure}
\begin{figure}
      \includegraphics[width=0.5\linewidth]{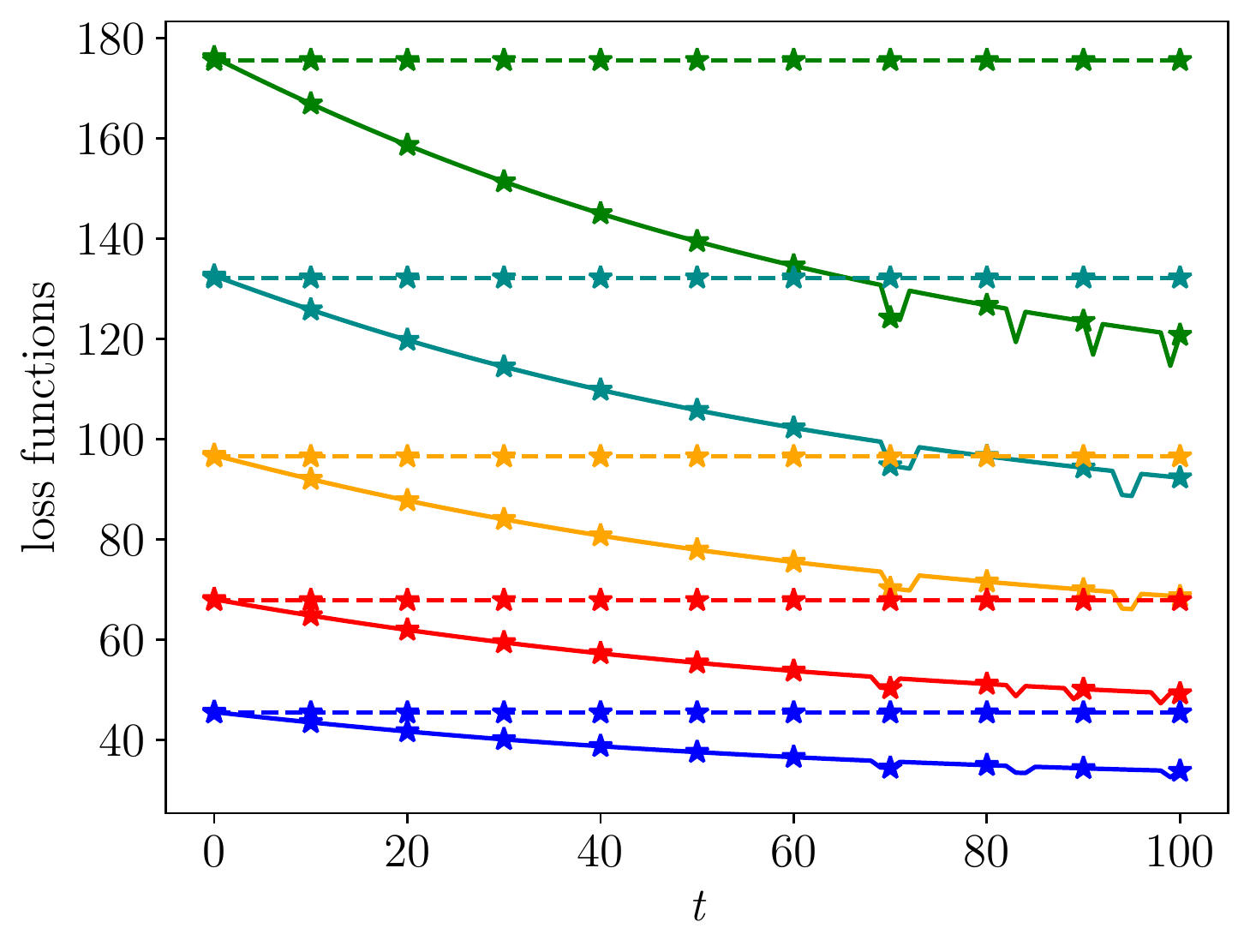}%
      \caption{The loss as function of the number of training iterations for
      $N=5,6,7,8,9$ (plotted with blue, red, orange, teal, and green).
      Dotted and solid lines correspond to
      approximate eigenvalues computed using Hermite and augmented Hermite discretization schemes, respectively.
      To improve the contrast between the lines, every 5\emph{th} marker is
      omitted.}
      \label{fig:loss_function}
\end{figure}

\begin{figure}
 \includegraphics[width=85mm]{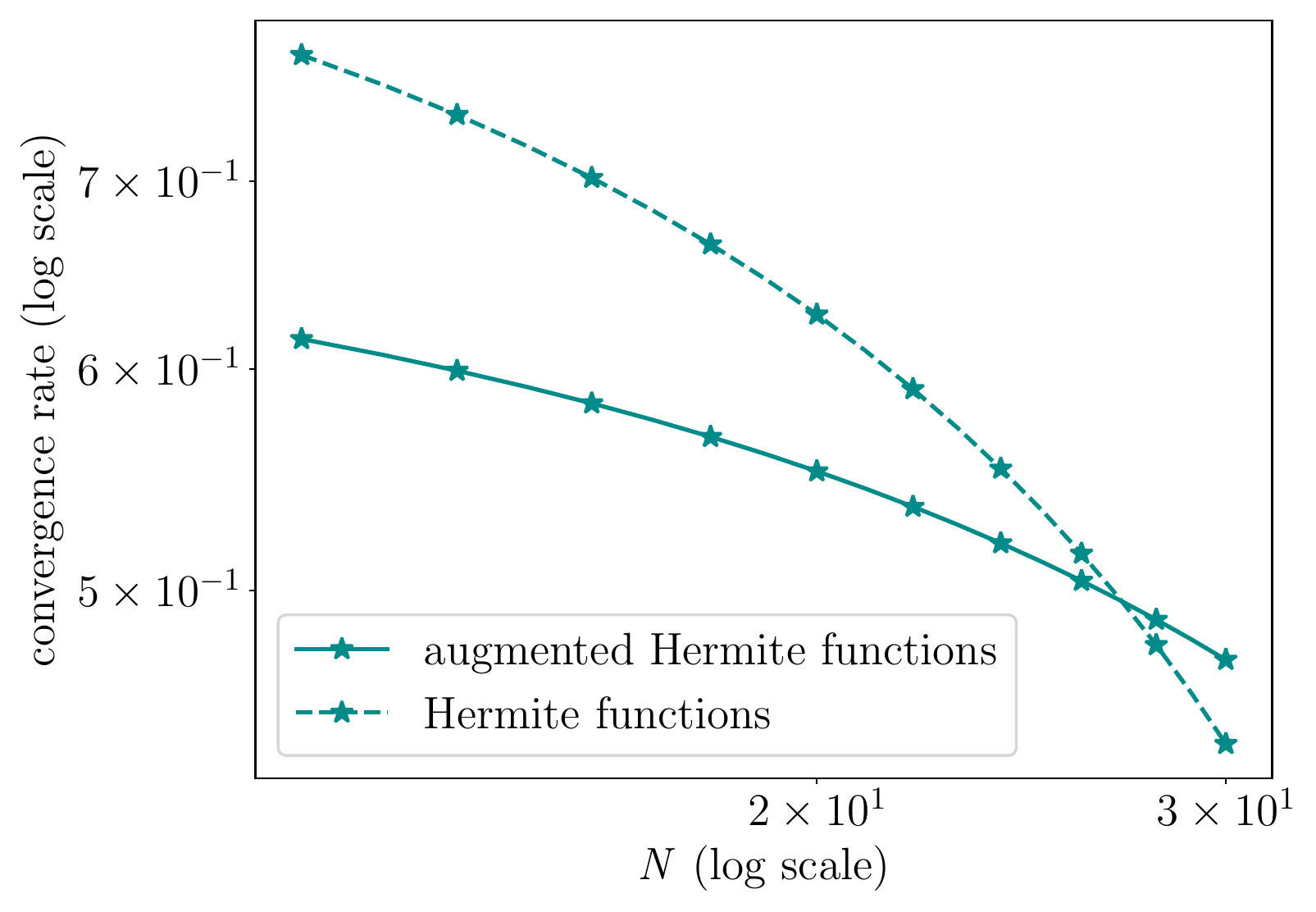}
   \caption{Convergence rate for the sum of the eigenvalues $\sum_{n=5}^{10}
   \tilde{E}_{n,N}$
   as a function of the basis truncation parameter $N$,
   computed using Hermite and augmented Hermite
   discretization schemes.
   To improve the contrast between the lines, every other marker is omitted.}
	\label{fig:convergence_rate}
\end{figure}

\section{Conclusions and Outlook}
\label{sec:conclusions}
Normalizing flows can be used to augment
the expressivity of standard basis sets of $L^2$. We show that augmented
basis sets define a rich approximation space consisting of a family of
parameterized linear spaces that are all dense in $L^2$.  We demonstrate the
concept's effectiveness by simulating the eigenfunctions of a perturbed quantum
harmonic oscillator using augmented Hermite functions. Our findings
indicate quick convergence of the numerical scheme with the size of the basis set.
Furthermore, the results were shown to be more accurate than those obtained by using
Hermite functions with the same basis size.

Despite the fact that numerical simulations have demonstrated the
convergence behavior
of augmented basis sets for simulating the eigenfunctions of Hamiltonian
operators as $N $ increases, theoretical results are still lacking.
Moreover, future work should focus on convergence for a fixed truncation
parameter $N$ by increasing the size of the normalizing flow.

This numerical scheme is yet to be used on other high-dimensional differential equations or on more
complicated quantum mechanical problems, \eg, simulations of dynamics and spectra of polyatomic
molecules, or quantum chemical calculations. In particular, the precision of standard basis sets for
approximating solutions of differential equations decreases exponentially as the number of
dimensions increase. While augmenting the basis sets with normalising flows reduces the size of the
linear expansion $N$ required to achieve a particular level of accuracy, it is unclear whether the
curse of dimensionality can be mitigated. Another issue in higher dimensions is the computation of
integrals. We used Gauss quadratures in this manuscript, which are not suitable for higher
dimensions because their size grows exponentially with the number of dimensions, though sparse grid
approaches, such as Smolyak grids, can mitigate the steep scaling for certain
problems~\cite{Avila:JCP139:134114}. Stochastic estimations of integrals, such as Monte-Carlo
methods, may provide a dimension-independent scaling at the expense of lower accuracy. Another
approach may be the use of collocation methods, which are equivalent to solving the Schrödinger
equation by demanding that it is satisfied at a set of points, \ie, no integration is
necessary~\cite{Yang:CPL153:98}.

We believe that augmented basis sets have the potential
for accurate modeling of excited states of quantum models with unbounded
potentials experiencing a dissociation behavior.

\section{Code Availability}
The computer code and all relevant data can be obtained from the online
repository at \href{https://github.com/CFEL-CMI/FlowBasis}{https://github.com/CFEL-CMI/FlowBasis}.

\section*{Supplementary Material}
\label{sec:supplementary}
The normalizing flow we used for training has the form
$$
    g^{-1}_\theta(x) =
    \tanh(f_ \gamma(\tanh^{-1}(x-\beta )/\alpha ))*\alpha+\beta,
$$
where $f$ is an invertible residual neural network, whose parameters are denoted by $\gamma$.
The neural network $f$ is composed of 1
layer with 128 hidden units, with Lipswish
activation functions, \ie, functions of the form
$$
    \sigma(x) = \frac{1}{1.1} \cdot \frac{x}{1+\exp(-x)}.
$$
Through a fixed scaling procedure, the input to the $\tanh^{-1}$ function is guaranteed to lie within $[-1,1]$.  The free parameters of
the normalizing flow $\theta=(\gamma, \alpha, \beta)$ are optimized using the Adam optimization algorithm in~\cite{Kingma:ICLR3:2015},
with a learning rate $\alpha=10^{-3}$. To compute the matrix elements of
$\tilde{H}$ in \eqref{eq:Schroedinger_weak}, we used Gauss-Hermite
quadratures of order $90$. This is possible and convenient since

\begin{equation*}
	\int \phi_i(g^{-1}(x)) V(x) \phi_j(g^{-1}(x)) |\det \nabla_x g^{-1}(x)| \ dx =
	\int \phi_i(x) V(g(x)) \phi_j(x)  \ dx,
\end{equation*}
\ie, matrix elements of the potential function in the augmented Hermite functions correspond to
matrix elements of the perturbed potential $V \circ g$ in the Hermite functions. While the integral
on the right hand side contains $V \circ g$, it is still possible to compute accurately with
Gauss-Hermite quadratures since the potential in our example is simple and $g$ is a smooth Lipschitz
function. Same argumentation can be perused for the kinetic energy operator. Note that the $n$th
Hermite functions have $n$ nodes, which means that high-order Hermite functions are harder to
integrate. Indeed, we observed that the calculated energies go below the 
variational limit when
using a larger basis set for a fixed order of quadrature points or when using a smaller order of
quadrature points for the same size of the basis. For the states considered in
\autoref{fig:convergence} we observed that using $90$ quadrature points was enough for obtaining
correct variational limits. Finally, note that the computation of quadrature 
points using numpy has only been tested up to degree
$100$. Higher degrees may be problematic, which restricted the size of
the basis set that we could consider in this study.

\section*{Acknowledgment}
   We thank Jannik Eggers and Vishnu Sanjay for useful comments and discussions in early stages of
   this work and Álvaro Fernández-Corral for useful comments on the manuscript.

   This work has been supported by Deutsches Elektronen-Synchtrotron DESY, a member of the Helmholtz
   Association (HGF), by the Data Science in Hamburg HELMHOLTZ Graduate School for the Structure of
   Matter (DASHH, HIDSS-0002), and by the Deutsche Forschungsgemeinschaft (DFG) through the cluster
   of excellence ``Advanced Imaging of Matter'' (AIM, EXC~2056, ID~390715994). We acknowledge the
   use of the Maxwell computational resources operated at Deutsches Elektronen-Synchrotron DESY,
   Hamburg, Germany.

\bibliography{string,cmi}
\end{document}